\documentclass{PoS}
\usepackage{amsmath,amssymb,amsfonts,amsthm}
\usepackage{mathrsfs}
\usepackage[arrow,matrix,curve]{xy}

\def\nn{\nonumber}
\def\sk{\vspace{1mm}}

\def\bbK{\mathbb{K}}

\def\bbC{\mathbb{C}}
\def\bbN{\mathbb{N}}

\def\FF{\mathcal{F}}
\def\RR{\mathcal{R}}

\def\MMM{\mathscr{M}}
\def\AAA{\mathscr{A}}

\newcommand{\MMMod}[3]{{}^{#1}{}_{#2}\MMM{\!}_{#3}}
\newcommand{\AAAlg}[3]{{}^{#1}{}_{#2}\AAA{\!}_{#3} }

\def\Hom{\mathrm{Hom}}
\def\End{\mathrm{End}}
\def\Con{\mathrm{Con}}

\def\id{\mathrm{id}}

\def\deg{\mathrm{deg}}
\def\dd{\mathrm{d}}
\def\nab{\nabla}

\def\cop{\mathrm{cop}}

\theoremstyle{plain}

\newtheorem{theo}{{\bf Theorem}}[section]

\newtheorem{lem}[theo]{{\bf Lemma}}

\newtheorem{propo}[theo]{{\bf Proposition}}

\theoremstyle{definition}

\newtheorem{defi}[theo]{{\bf Definition}}

\newtheorem{ex}[theo]{{\bf Example}}

\theoremstyle{remark}

\newtheorem{rem}[theo]{{\bf Remark}}

\def\nab{\nabla}
\def\ra{\triangleright}
\def\RA{\blacktriangleright}

\title{Twist deformations of module homomorphisms and connections \thanks{Paper based on joint work with Paolo Aschieri}}

\ShortTitle{Twist deformations of module homomorphisms and connections}

\author{\speaker{Alexander Schenkel}\\
      Fachgruppe Mathematik\\
Bergische~Universit\"at~Wuppertal,~Gau\ss stra\ss e~20,~42119~Wuppertal,~Germany\\
       E-mail: \email{schenkel@math.uni-wuppertal.de}}

\abstract{
Let $H$ be a Hopf algebra, $A$ a left $H$-module algebra and $V$ a left $H$-module $A$-bimodule.
 We study the behavior of the right $A$-linear endomorphisms of $V$ under twist deformation.
 We in particular construct a bijective quantization map to the right $A_\star$-linear endomorphisms
 of $V_\star$,  with $A_\star,V_\star$ denoting the usual twist deformations of $A,V$.
 The quantization map is extended to right $A$-linear homomorphisms between 
 left $H$-module $A$-bimodules and to right connections on $V$.
 We then investigate the tensor product of linear maps between left $H$-modules. Given a quasitriangular Hopf algebra
 we can define an $H$-covariant tensor product of linear maps, which restricts for left $H$-module $A$-bimodules
 to a well-defined tensor product of right $A$-linear homomorphisms on tensor product modules over $A$. This also
 requires a quasi-commutativity condition on the algebra and bimodules.
 Using this tensor product we can construct a new lifting prescription of connections to tensor product
  modules, generalizing the usual prescription to also include nonequivariant connections.
 
}

\FullConference{Proceedings of the Corfu Summer Institute 2011 "School and Workshops on Elementary Particle Physics and Gravity"\\
		 September 4-18, 2011\\
		 Corfu, Greece}

\begin{document}


\section{\label{sec:intro}Introduction}
In this proceedings article I will report on a  joint work with Paolo Aschieri \cite{ASCHIERISCHENKEL}
aimed at understanding better the mathematical structures behind noncommutative gravity.
Our work is based on the approach initiated by 
Julius Wess and his group \cite{Aschieri:2005yw,Aschieri:2005zs}, in which a particular emphasis is given to
the deformed Hopf algebra of diffeomorphisms.
For exact solutions of the noncommutative Einstein equations see \cite{Schupp:2009pt,Ohl:2009pv,Aschieri:2009qh} 
and for quantum field theory on noncommutative curved spacetimes based on this formalism
see \cite{Ohl:2009qe,Schenkel:2010sc,Schenkel:2010jr,Schenkel:2011gw}.
Other recent mathematical developments in noncommutative (and also nonassociative) differential geometry and 
Riemannian geometry  can be found in the papers of Beggs and Majid \cite{BeggsMajid1,BeggsMajid2}.

Let us briefly review the basic idea of the approach of Wess et al.~\cite{Aschieri:2005yw,Aschieri:2005zs}:
The starting point is a smooth manifold $M$, which will be later subject to deformation quantization.
Associated to $M$ there is the  algebra $C^\infty(M)$ of smooth complex-valued functions
and the $C^\infty(M)$-modules $\Xi := \Gamma^\infty(TM)$ and $\Omega:=\Gamma^\infty(T^\ast M)$
of smooth complex-valued sections of the tangent bundle $TM$ and cotangent bundle $T^\ast M$, respectively. In fact,
$\Xi$ and $\Omega$ are $C^\infty(M)$-bimodules and this bimodule structure
is essential for considering tensor fields, which are tensor products over $C^\infty(M)$ of $C^\infty(M)$-bimodules.
The infinitesimal diffeomorphisms of $M$ are described by
the Lie algebra of vector fields $(\Xi,[\cdot,\cdot])$, or equivalently by the universal enveloping algebra
$U\Xi$, carrying a canonical Hopf algebra structure. The Hopf algebra $U\Xi$ acts canonically 
via the Lie derivative on $C^\infty(M)$, $\Xi$ and $\Omega$, giving $C^\infty(M)$ the structure of a left
$U\Xi$-module algebra and $\Xi$, $\Omega$ the structure of a left $U\Xi$-module $C^\infty(M)$-bimodule.
Based on this algebraic formulation of classical differential geometry, the basic 
idea of \cite{Aschieri:2005yw,Aschieri:2005zs} is to use well-known Hopf algebra
deformation methods (see e.g.~\cite{Majid:1996kd}) to go over to the realm of noncommutative 
geometries. In short,
the Hopf algebra of diffeomorphisms is deformed by a  Drinfel'd twist $\FF$ \cite{Drinfeld} resulting
in a new Hopf algebra $U\Xi^\FF$.
Demanding covariance under $U\Xi^\FF$ requires us to deformed all 
spaces the Hopf algebra $U\Xi$ acts on,
in particular $C^\infty(M)$, $\Xi$ and $\Omega$ are respectively deformed
into a left $U\Xi^\FF$-module algebra $C^\infty(M)_\star$ and  left $U\Xi^\FF$-module $C^\infty(M)_\star$-bimodules
$\Xi_\star$, $\Omega_\star$. This deformation induces a quantization in
the sense that, in general, the resulting deformed algebra $C^\infty(M)_\star$ is not commutative. 
Based on these basic deformed objects a noncommutative gravity theory is constructed, in particular, 
covariant derivatives, metrics, torsion and curvature are introduced. 

As a mathematical abstraction of the setting above we may  study general Hopf algebras
$H$ and general left $H$-module algebras $A$, which we shall interpret as noncommutative symmetries and spaces. 
Noncommutative vector bundles (equipped with a left $H$-action) over these noncommutative spaces
are then described by left $H$-module $A$-bimodules.
There are different choices for morphisms between noncommutative vector bundles:
We could in principle consider the set of morphisms respecting all module structures, i.e.~left $H$-module $A$-bimodule 
homomorphisms. This choice is inconvenient for gravity theories, since if we interpret the metric field 
as a morphism from the module of vector fields to its dual (the module one-forms), it
 is certainly not a $U\Xi$-equivariant map and thus not part of this set of morphisms. A convenient and 
 much more flexible alternative is to consider the set of right $A$-linear homomorphisms between left $H$-module 
$A$-bimodules (equivalently we could also consider left $A$-linear homomorphisms). 
Let us use the short notation $\Hom_A$ for right $A$-linear homomorphisms between  left $H$-module 
$A$-bimodules.

Given the fact that the noncommutative gravity theory in \cite{Aschieri:2005yw,Aschieri:2005zs}
 is obtained by deforming classical gravity it is interesting to study the behavior of
 $\Hom_A$-morphisms under twist deformations. On one hand, we can apply general Hopf algebra
 deformation methods in order to deform $\Hom_A$-morphisms and on the other hand
 we can consider $\Hom_{A_\star}$-morphisms on the deformed left $H^\FF$-module
 $A_\star$-bimodules. We relate these a priori different constructions via a structure preserving
 isomorphism, thus showing that up to isomorphism there is only one twist quantization of
 $\Hom_A$-morphisms. A consequence of this isomorphism is a quantization map, yielding
 a bijective correspondence between $\Hom_A$ and $\Hom_{A_\star}$-morphisms.
 As a remark, our isomorphism agrees with the one recently found by Kulish and Mudrov
\cite{Kulish:2010mr}, which is a generalization of the isomorphism $D$ in \cite{Aschieri:2005zs}.
 
 In addition to $\Hom_A$-morphisms also connections on noncommutative vector bundles
 play a prominent role in noncommutative gravity. Again, assuming these connections to be compatible
 with the complete left $H$-module $A$-bimodule structure is too restrictive. Because of this
 we make the convenient choice of connections compatible with only the right $A$-module structure,
 i.e.~connections satisfying a Leibniz rule with respect to the right $A$-action.
 We use the short notation $\Con_A$ for these connections. We also construct a quantization map
 in this case, yielding a bijective correspondence between $\Con_A$ and $\Con_{A_\star}$-connections,
 the latter satisfying a Leibniz rule with respect to the right $A_\star$-action on the
 deformed left $H^\FF$-module $A_\star$-bimodules.

Given $\Hom_A$-morphisms and $\Con_A$-connections
on left $H$-module $A$-bimodules it is important to understand
their lifts to tensor product modules. Since our homomorphisms and
connections are not assumed to be $H$-equivariant, these lifting prescriptions
become more complicated and we require further structures.
For quasitriangular Hopf algebras we have a braiding isomorphism
on tensor products of left $H$-modules allowing us to define a
tensor product of linear maps, which is compatible with the $H$-action.
For $H$-equivariant maps this tensor product reduces to the usual tensor product
of linear maps.
Our tensor product restricts to a well-defined tensor product of
$\Hom_A$-morphisms on tensor product modules over $A$, if
we assume $A$ and the left $H$-module $A$-bimodules to be quasi-commutative.
Similarly, we can define a sum of $\Con_A$-connections giving a prescription
to associate to two $\Con_A$-connections a $\Con_A$-connection on the tensor product module over $A$.
This also requires a quasitriangular Hopf algebra and quasi-commutative algebras and bimodules.
For the special case of $H$-equivariant connections we recover the well-known tensor product
connection of \cite{Mourad:1994xa,DuViMasson,DuViLecture}. In this sense we provide a generalization
of the usual lifting formula, which is also applicable to non $H$-equivariant
connections. In the language of \cite{Mourad:1994xa,DuViMasson,DuViLecture} we can say that
our sum of connections can not only be applied to right bimodule connections, but also to generic
right connections on bimodules.

The organization of this paper is as follows: 
In Section \ref{sec:basics} we fix the notation and review some basics on modules, algebras, Hopf algebras and
 twist deformations required for the main text. Extensive reviews can be found for example in the monographs 
\cite{Majid:1996kd,Kassel:1995xr}.
In Section \ref{sec:hom} we focus on the twist deformation of
module homomorphisms and endomorphisms. In particular we provide a quantization map
from $\Hom_A$-morphisms to $\Hom_{A_\star}$-morphisms. The twist deformation of connections is discussed
 in Section \ref{sec:con}  and a quantization map is explicitly constructed.
In Section \ref{sec:homlift} we provide a covariant lifting prescription for homomorphisms
to tensor product modules. For quasi-commutative algebras and bimodules we
show that the lifting prescription induces a well-defined tensor product of $\Hom_A$-morphisms 
on tensor product modules over $A$.
The sum of $\Con_A$-connections is discussed in Section \ref{sec:conlift}.


\section{\label{sec:basics}Preliminaries on modules, algebras, Hopf algebras and twist deformations}
\subsection{Algebraic preliminaries}
We fix the notation and introduce the algebraic structures relevant for this article.
We denote by $\bbK$ a fixed commutative and unital ring. This in particular includes the case of $\bbK =\bbC$ used in 
commutative differential geometry and $\bbK=\bbC[[h]]$ (the formal power series extension of
 the complex numbers in an indeterminate $h$) used in deformation quantization.
 
 A {\bf $\bbK$-module} is an abelian group $A$ with an action $\bbK\times A\to A\,,~(\lambda,a)\mapsto\lambda\,a$,
  such that, for all $\lambda,\tilde\lambda\in\bbK$ and $a,b\in A$,
 \begin{flalign}
 (\lambda\,\tilde\lambda)\,a = \lambda\,(\tilde\lambda a)~,\quad \lambda\,(a+b)=\lambda\,a+\lambda\,b~,\quad (\lambda+\tilde\lambda)\,a = \lambda\,a +\tilde\lambda\,a ~,\quad 1\,a = a~.
 \end{flalign}
A {\bf $\bbK$-module homomorphism} (or {\bf $\bbK$-linear map}) between the $\bbK$-modules
$A$ and $B$ is a homomorphism $\varphi:A\to B$ of the abelian groups, such that, for all $\lambda\in\bbK$ and $a\in A$,
$\varphi(\lambda\,a) = \lambda\,\varphi(a)$. 

An {\bf algebra} is a $\bbK$-module $A$  with a $\bbK$-linear map $\mu:A\otimes A\to A$ (product).
Here and in the following $\otimes$ is the tensor product over $\bbK$.
We denote by $a\otimes b$ the image of $(a,b)$ under the canonical $\bbK$-bilinear map $A\times A\to A\otimes A$
and simply write for the product $\mu(a\otimes b) = a\,b$. The algebra $A$ is associative if, for all $a,b,c\in A$,
$a\,(b\,c) = (a\,b)\,c$ and unital if there exists a unit element $1\in A$, such that, for all $a\in A$, $1\,a = a\,1 = a$.
If not stated otherwise, algebras will always be associative and unital.

\begin{defi}\label{defi:HA}
A {\bf Hopf algebra}  is an associative and unital algebra $H$ together with two algebra homomorphisms
$\Delta: H\to H\otimes H$ (coproduct), $\epsilon: H\to \bbK$ (counit) and a $\bbK$-linear map
$S:H\to H$ (antipode), such that, for all $\xi\in H$,
\begin{subequations}\label{eqn:HAprop}
\begin{flalign}
(\Delta\otimes\id)\Delta(\xi) &= (\id\otimes\Delta)\Delta(\xi)~,\quad\text{(coassociativity)}\\
(\epsilon\otimes\id)\Delta(\xi) &= (\id\otimes\epsilon)\Delta(\xi) = \xi~,\\
\mu\big((S\otimes\id)\Delta(\xi)\big) &= \mu\big((\id\otimes S)\Delta(\xi)\big) = \epsilon(\xi)\,1~.
\end{flalign}
\end{subequations}
The product in $H\otimes H$ is defined by, for all $\xi,\zeta,\tilde\xi,\tilde\zeta\in H$,
\begin{flalign}
(\xi\otimes\zeta) \,(\tilde\xi\otimes\tilde\zeta) = (\xi\,\tilde\xi)\otimes(\zeta\,\tilde\zeta)~.
\end{flalign}
\end{defi}
\sk

It is convenient to introduce a short notation (Sweedler's notation) for the coproduct,
for all $\xi\in H$, $\Delta(\xi) = \xi_1\otimes\xi_2$ (sum understood). 
The Hopf algebra properties (\ref{eqn:HAprop}) then read
\begin{subequations}
\begin{flalign}
\xi_{1_1}\otimes\xi_{1_2} \otimes \xi_2 &= \xi_1\otimes\xi_{2_1}\otimes\xi_{2_2}=:\xi_1\otimes\xi_2\otimes\xi_3~,\\
\epsilon(\xi_1)\, \xi_2 &= \xi_1\,\epsilon(\xi_2)  = \xi~,\\
S(\xi_1)\,\xi_2 &= \xi_1\,S(\xi_2) = \epsilon(\xi)\,1~.
 \end{flalign} 
\end{subequations}
Likewise we denote the three times iterated application of $\Delta$ on $\xi$
by $\xi_1\otimes\xi_2\otimes\xi_3\otimes\xi_4$.
Further standard properties of the antipode which follow from Definition \ref{defi:HA}
are (see e.g.~\cite{Majid:1996kd}), for all $\xi,\zeta\in H$,
$S(\xi\,\zeta) = S(\zeta)\,S(\xi)$, $S(1)=1$, $S(\xi_1)\otimes S(\xi_2) = S(\xi)_2\otimes S(\xi)_1$
and $\epsilon(S(\xi)) = \epsilon(\xi)$.
 
 \begin{defi}
 A {\bf left module over an algebra $A$} (or a {\bf left $A$-module}) is a $\bbK$-module
 $V$ together with a $\bbK$-linear map $\cdot : A\otimes V\to V$ satisfying, for all $a,b\in A$ and $v\in V$,
 \begin{flalign}
 a\cdot (b\cdot v) = (a\,b)\cdot v~,\quad 1\cdot v = v~.
 \end{flalign}
 The class of left $A$-modules is denoted by $\MMMod{}{A}{}$.
 
 Analogously, a {\bf right $A$-module} is a $\bbK$-module  $V$ with a $\bbK$-linear map
 $\cdot:V\otimes A \to V$ satisfying, for all $a,b\in A$ and $v\in V$,
 \begin{flalign}
( v\cdot a)\cdot b = v\cdot(a\,b)~,\quad v\cdot 1 =v~.
 \end{flalign} 
 The class of right $A$-modules is denoted by $\MMMod{}{}{A}$.
 
 An {\bf $A$-bimodule} is a left and a right $A$-module $V$
 satisfying the compatibility condition, for all $a,b\in A$ and $v\in V$,
 \begin{flalign}
 (a\cdot v)\cdot b = a\cdot (v\cdot b)~.
 \end{flalign}
 The class of $A$-bimodules is denoted by $\MMMod{}{A}{A}$.
 \end{defi}
\sk

The algebra $A$ can itself be a module over another algebra $H$. If $H$ is further a Hopf algebra, 
we have the notion of an $H$-module algebra, expressing a covariant transformation behavior of
$A$ under $H$.
\begin{defi}
Let $H$ be a Hopf algebra. A {\bf left $H$-module algebra} is an algebra $A$ that is also
a left $H$-module (with $H$-action denoted by $\ra$), such that, for all $a,b\in A$ and $\xi\in H$,
\begin{flalign}
\xi\ra (a\,b) =( \xi_1\ra a)\,(\xi_2\ra b)~,\quad \xi\ra 1 =\epsilon(\xi)\,1~.
\end{flalign}
The class of left $H$-module algebras is denoted by $\AAAlg{H}{}{}$.

Let $A\in\AAAlg{H}{}{}$. A {\bf left $H$-module $A$-bimodule} is
an $A$-bimodule $V$ that is also a left $H$-module (with $H$-action denoted by $\ra$), such that, 
for all $a\in A$, $v\in V$ and $\xi\in H$,
\begin{flalign}\label{eqn:leftHmoduleABbimodule}
\xi\ra (a\cdot v) = (\xi_1\ra a)\cdot (\xi_2\ra v)~,\quad \xi\ra(v\cdot a) = (\xi_1\ra v)\cdot (\xi_2\ra a)~.
\end{flalign} 
The class of left $H$-module $A$-bimodules is denoted by $\MMMod{H}{A}{A}$.

An algebra $E$ is a {\bf left $H$-module $A$-bimodule algebra}, if $E$ is a left $H$-module $A$-bimodule
and also a left $H$-module algebra. The class of left $H$-module $A$-bimodule algebras is denoted by
$\AAAlg{H}{A}{A}$.
\end{defi}
\sk

The classes $\MMMod{H}{A}{}$ and $\MMMod{H}{}{A}$ are defined analogously by restricting
(\ref{eqn:leftHmoduleABbimodule}) respectively to the first or second condition.
 
\begin{ex}\label{ex:classicalgeometry}
Consider the universal enveloping algebra $U\Xi$ associated to the Lie algebra of vector fields $\Xi$
on a smooth manifold $M$. $U\Xi$ has a natural Hopf algebra structure as explained e.g.~in \cite{Aschieri:2005zs}.

The space of vector fields $\Xi$ (and also the space of one-forms $\Omega$) is a bimodule over $C^\infty(M)$, 
i.e.~$\Xi,\Omega\in \MMMod{}{C^\infty(M)}{C^\infty(M)}$. The vector fields $\Xi$ act as derivations on $C^\infty(M)$,
turning $C^\infty(M)$ canonically into a left $U\Xi$-module algebra, i.e.~$C^\infty(M)\in \AAAlg{U\Xi}{}{}$.
The Lie derivative $\mathcal{L}$ on $\Xi$ and $\Omega$ turns these modules canonically into left $U\Xi$-module
$C^\infty(M)$-bimodules, i.e.~$\Xi,\Omega\in\MMMod{U\Xi}{C^\infty(M)}{C^\infty(M)}$.

The twist deformation procedure explained in the next subsection provides the examples relevant for noncommutative
gravity.
\end{ex}


\subsection{Twist deformations}
\begin{defi}
Let $H$ be a Hopf algebra. A {\bf twist} $\FF$ is an element $\FF\in H\otimes H$
that is invertible and that satisfies 
\begin{subequations}\label{eqn:twistprop}
\begin{flalign}
\label{eqn:twistprop1}\FF_{12}\,(\Delta\otimes\id) \FF = \FF_{23}\,(\id\otimes \Delta)\FF~,\quad\text{(2-cocycle property)}\\
\label{eqn:twistprop2}(\epsilon\otimes\id)\FF = 1= (\id\otimes\epsilon)\FF ~,\quad\text{(normalization property)}
\end{flalign}
\end{subequations}
where $\FF_{12}=\FF\otimes 1$ and $\FF_{23} = 1\otimes\FF$.
\end{defi}
\sk
We frequently use the notation (sum over $\alpha$ understood)
\begin{flalign}\label{eqn:twistnot}
\FF=f^\alpha\otimes f_\alpha~,\quad \FF^{-1} = \bar f^\alpha\otimes\bar f_\alpha~.
\end{flalign}
The properties (\ref{eqn:twistprop}) and the inverse of (\ref{eqn:twistprop1}) read in this notation
\begin{subequations}
\begin{flalign}
f^\beta f^\alpha_1\otimes f_\beta f^\alpha_2 \otimes f_\alpha &= f^\alpha\otimes f^\beta f_{\alpha_1}\otimes f_\beta f_{\alpha_2}~,\\
\epsilon(f^\alpha)\,f_\alpha &= 1 = f^\alpha\,\epsilon(f_\alpha)~,\\
\bar f^\alpha_1\bar f^\beta\otimes\bar f^\alpha_2\bar f_\beta \otimes \bar f_\alpha &= \bar f^\alpha\otimes\bar f_{\alpha_1}\bar f^\beta \otimes\bar f_{\alpha_2}\bar f_\beta~.
\end{flalign}
\end{subequations}

We now recall how a twist $\FF$ of a Hopf algebra $H$ induces a deformation $H^\FF$ of the Hopf algebra and all its
$H$-modules that become $H^\FF$-modules. In particular, $H$-module algebras are deformed into $H^\FF$-module algebras
and commutative ones are typically deformed into noncommutative ones. In this sense $\FF$ induces a quantization.
\begin{theo}\label{theo:HAdef}
Given a Hopf algebra $H$ and a twist $\FF\in H\otimes H$ there is a new Hopf algebra $H^\FF$ given by
\begin{flalign}
\big(H,\mu,\Delta^\FF,\epsilon,S^\FF\big)~.
\end{flalign}
As algebras $H^\FF=H$ and they also have the same counit $\epsilon^\FF=\epsilon$. The deformed coproduct is given by,
for all $\xi\in H$,
\begin{flalign}
\Delta^\FF(\xi) = \FF\,\Delta(\xi)\,\FF^{-1}~. 
\end{flalign}
The deformed antipode reads, for all $\xi\in H$,
\begin{flalign}
S^\FF(\xi) = \chi\,S(\xi)\,\chi^{-1}~,
\end{flalign}
where $\chi := f^\alpha \,S(f_\alpha)$ and $\chi^{-1}=S(\bar f^\alpha)\,\bar f_\alpha$.
\end{theo}
A proof of this theorem can be found in textbooks on Hopf algebras, e.g.~\cite{Majid:1996kd}.
\begin{rem}
The Hopf algebra $H^\FF$ admits the twist $\FF^{-1}$. The deformed Hopf algebra
$(H^\FF)^{\FF^{-1}}$ is equal to the original one $H$. Because of this we call the deformation
$\FF^{-1}$ of the Hopf algebra $H^\FF$ dequantization. 
\end{rem}
\begin{theo}\label{theo:algmoddef}
Given a Hopf algebra $H$, a twist $\FF\in H\otimes H$ and a left $H$-module algebra $A\in\AAAlg{H}{}{}$,
then there is a deformed left $H^\FF$-module algebra $A_\star\in\AAAlg{H^\FF}{}{}$. 
As a left $H$-module (and thus also as a left $H^\FF$-module, since $H$ and $H^\FF$ are equal as algebras) $A_\star$ is equal to $A$. The product 
in $A_\star$ is given by, for all $a,b\in A$,
\begin{flalign}\label{eqn:starprod}
a\star b := \mu_\star(a\otimes b) = \mu\circ\FF^{-1}\ra (a\otimes b)= (\bar f^\alpha\ra a)\,(\bar f_\alpha\ra b)~.
\end{flalign}

Given  also a left $H$-module $A$-bimodule $V\in\MMMod{H}{A}{A}$,
then there is a deformed left $H^\FF$-module $A_\star$-bimodule $V_\star\in\MMMod{H^\FF}{A_\star}{A_\star}$.
As a left $H$-module (and thus also as a left $H^\FF$-module) $V_\star$ is equal to $V$. The left and right
 $A_\star$-actions are given by,
for all $a\in A$ and $v\in V$,
\begin{subequations}\label{eqn:starmodprod}
\begin{flalign}
\label{eqn:starmodprod1}a\star v = \cdot \circ\FF^{-1}\ra (a\otimes v) = (\bar f^\alpha\ra a)\cdot (\bar f_\alpha\ra v)~,\\
\label{eqn:starmodprod2}v\star a = \cdot \circ \FF^{-1}\ra (v\otimes a) = ( \bar f^\alpha\ra v)\cdot (\bar f_\alpha\ra a)~.
\end{flalign}
\end{subequations}

Given also a left $H$-module $A$-bimodule algebra $E\in\AAAlg{H}{A}{A}$, then there is a deformed
left $H^\FF$-module $A_\star$-bimodule algebra $E_\star\in\AAAlg{H^\FF}{A_\star}{A_\star}$. As a
left $H$-module (and thus also as a left $H^\FF$-module) $E_\star$ is equal to $E$, the product is given by (\ref{eqn:starprod}) and 
the $A_\star$-actions by (\ref{eqn:starmodprod}).
\end{theo}
For a proof of this theorem we refer to \cite{GIAQUINTO,ASCHIERISCHENKEL}. See also \cite{SCHENKELDISS} 
for an early account of the results in collaboration \cite{ASCHIERISCHENKEL}.
  
Analogously to Theorem \ref{theo:algmoddef} we can deform $\MMMod{H}{A}{}$ and $\MMMod{H}{}{A}$
modules into $\MMMod{H^\FF}{A_\star}{}$ and $\MMMod{H^\FF}{}{A_\star}$ modules by restricting
(\ref{eqn:starmodprod}) to (\ref{eqn:starmodprod1}) and (\ref{eqn:starmodprod2}), respectively.
\begin{ex}
Deforming via the theorems above the Hopf algebra $U\Xi$ and its modules $C^\infty(M)$, $\Xi$ and $\Omega$ (see Example 
\ref{ex:classicalgeometry}) we obtain the deformed diffeomorphisms $U\Xi^\FF$ and its deformed modules $C^\infty(M)_\star$,
$\Xi_\star$ and $\Omega_\star$. These structures are the building blocks for noncommutative gravity
 \cite{Aschieri:2005yw,Aschieri:2005zs}.
\end{ex}

\section{\label{sec:hom}Quantization of module homomorphisms}
\subsection{Endomorphisms}
We  investigate the quantization of the endomorphisms of a module $V$.
For this let us consider first the case of a left $H$-module $V\in\MMMod{H}{}{}$ and later introduce 
an additional $A$-bimodule structure on $V$. 
The set of all $\bbK$-linear maps from $V$ to itself is denoted by $\End_\bbK(V)$. We can compose 
two endomorphisms $P,Q\in\End_\bbK(V)$ with the usual composition product $\circ$, 
i.e.~for all $v\in V$, $P\circ Q(v):= P(Q(v))$. This turns $\End_\bbK(V)$ into an algebra.
Furthermore, we have an induced left $H$-action on $\End_\bbK(V)$ given by the adjoint action, for all
$\xi\in H$ and $P\in \End_\bbK(V)$,
\begin{flalign}\label{eqn:adjointaction}
\xi\RA P := \xi_1\ra\circ P\circ S(\xi_2)\ra~.
\end{flalign} 
In this expression we have denoted, for all $\xi\in H$, by $\xi\ra\in \End_\bbK(V)$ the endomorphism $v\mapsto \xi\ra v$.
The product $\circ$ and the adjoint action $\RA$ are compatible, for all $\xi\in H$ and $P,Q\in\End_\bbK(V)$,
\begin{flalign}
\nn \xi\RA (P\circ Q) &= \xi_1\ra\circ P\circ Q\circ S(\xi_2)\ra = \xi_1\ra\circ P\circ S(\xi_2)\ra\circ\xi_3\ra\circ Q\circ S(\xi_4)\ra~\\
&=(\xi_1\RA P)\circ (\xi_2\RA Q)~,
\end{flalign}
where we have used the Hopf algebra properties (\ref{eqn:HAprop}). Thus,
$\End_\bbK(V)\in \AAAlg{H}{}{}$ is a left $H$-module algebra.

If additionally $V\in \MMMod{H}{}{A}$ is a left $H$-module right $A$-module we can consider the subset 
of  right $A$-linear endomorphisms $\End_A(V)\subseteq \End_\bbK(V)$, i.e.~$P\in\End_A(V)$ satisfies,
for all $a\in A$ and $v\in V$, $P(v\cdot a) = P(v)\cdot a$. These endomorphisms form a left $H$-module
subalgebra of $\End_\bbK(V)$, i.e.~$\End_A(V)\in\AAAlg{H}{}{}$, with the usual composition product $\circ$ and
the adjoint action (\ref{eqn:adjointaction}).

Finally, if $V\in\MMMod{H}{A}{A}$ is a left $H$-module $A$-bimodule, then
the left $A$-action on $V$ induces an $A$-bimodule structure on
 $\End_A(V)$ and also on $\End_\bbK(V)$, for all $a\in A$ and $P\in \End_\bbK(V)$,
\begin{flalign}\label{eqn:homAaction}
a\cdot P := l_a\circ P~,\quad P\cdot a:= P\circ l_a~,
\end{flalign}
where, for all $a\in A$, $l_a:V\to V\,,~v\mapsto a\cdot v$ is the left multiplication map.
Thus, $\End_\bbK(V),\End_A(V)\in\AAAlg{H}{A}{A}$ are left $H$-module $A$-bimodule algebras (the covariance
of the $A$-bimodule structure under $H$ is an immediate consequence of the left $H$-module algebra homomorphism 
$A\to \End_\bbK(V)\,,~a\mapsto l_a$). 

It is also possible to consider left $A$-linear endomorphisms $_A\End(V)$ of the module $V\in\MMMod{H}{A}{A}$
and their deformation. The corresponding results follow from the discussion of right $A$-linear endomorphisms
and a mirror construction, see \cite{ASCHIERISCHENKEL,SCHENKELDISS}. Since this is rather technical we restrict ourselves
in this proceedings article to the right linear case $\End_A(V)$.

Let $H$ be a Hopf algebra with twist $\FF\in H\otimes H$. We consider from now on only the case of
$V\in\MMMod{H}{A}{A}$, since the results for modules with only a left or right $A$-module structure
 follow by simply forgetting the unwanted structures.
 We are interested in the deformation of the endomorphism algebras $\End_\bbK(V), \End_A(V)\in\AAAlg{H}{A}{A}$.
 There are two options: Firstly, we can apply Theorem \ref{theo:algmoddef} to the left 
 $H$-module $A$-bimodule algebras $\End_\bbK(V), \End_A(V)\in\AAAlg{H}{A}{A}$.
 The results are the left $H^\FF$-module $A_\star$-bimodule algebras 
 $\End_\bbK(V)_\star, \End_A(V)_\star\in\AAAlg{H^\FF}{A_\star}{A_\star}$. There, the Hopf algebra 
 action is not deformed, the product is given by the $\star$-composition
 \begin{flalign}
 P\circ_\star Q = (\bar f^\alpha\RA P)\circ (\bar f_\alpha\RA Q)~,
 \end{flalign}
 and similarly the $A_\star$-bimodule structure is given by
 \begin{flalign}
 a\star P = (\bar f^\alpha\ra a)\cdot (\bar f_\alpha\RA P)~,\quad P\star a = (\bar f^\alpha\RA P)\cdot (\bar f_\alpha\ra a)~.
 \end{flalign}
The elements of $\End_A(V)_\star$ are right $A$-linear and in general {\it not} right $A_\star$-linear.
The second option is to consider the endomorphisms of the deformed module $V_\star \in\MMMod{H^\FF}{A_\star}{A_\star}$,
i.e.~$\End_\bbK(V_\star),\End_{A_\star}(V_\star)\in\AAAlg{H^\FF}{A_\star}{A_\star}$.
Here the Hopf algebra action is the $H^\FF$-adjoint action, for all $\xi\in H^\FF$ and $P\in \End_\bbK(V_\star)$,
\begin{flalign}
\xi\RA_\FF P := \xi_{1_\FF}\ra\circ P\circ S^\FF(\xi_{2_\FF})\ra~,
\end{flalign}
with the deformed coproduct $\Delta^\FF(\xi)=\xi_{1_\FF}\otimes \xi_{2_\FF}$ and the deformed antipode $S^\FF$ (see Theorem
\ref{theo:HAdef}). The product is the usual composition $\circ$ and the $A_\star$-bimodule structure is 
induced by the left $H^\FF$-module algebra homomorphism
$A_\star\to \End_\bbK(V_\star)\,,~a\mapsto l^\star_a$ given by, for all $v\in V_\star$, $l^\star_a(v)=a\star v$,
i.e.~for all $P\in \End_\bbK(V_\star)$ and $a\in A_\star$,
\begin{flalign}
a\ast P = l_a^\star\circ P~,\quad P\ast a = P\circ l^\star_a~.
\end{flalign}
The following theorem, which is proven in  \cite{ASCHIERISCHENKEL,SCHENKELDISS}, states 
that there is an isomorphism between both options:
\begin{theo}\label{theo:quantendo}
The map 
\begin{flalign}
D_\FF: \End_\bbK(V)_\star \to \End_\bbK(V_\star)~,~~P\mapsto D_\FF(P) = ( \bar f^\alpha\RA P)\circ \bar f_\alpha\ra~
\end{flalign}
is a  left $H^\FF$-module $A_\star$-bimodule algebra isomorphism.
It restricts to a left $H^\FF$-module $A_\star$-bimodule algebra isomorphism
\begin{flalign}
D_\FF:\End_A(V)_\star \to \End_{A_\star}(V_\star)~.
\end{flalign}
\end{theo}
Notice that the elements of $\End_A(V)_\star$ are ``classical'' endomorphisms, i.e.~right linear with respect to the
undeformed algebra $A$, while the elements of $\End_{A_\star}(V_\star)$ are right linear with respect to the deformed
algebra $A_\star$. Thus, we also have found a quantization map providing a bijective correspondence
 between classical and deformed right linear maps.

A similar map $D_\FF$ also exists for a variety of other examples, see \cite{ASCHIERISCHENKEL,SCHENKELDISS}.


\subsection{Homomorphisms}
The results of the last subsection naturally generalize to module homomorphisms.
To start, let us clarify the relevant algebraic structures. 
Consider two left $H$-modules $V,W\in\MMMod{H}{}{}$. The space $\Hom_\bbK(V,W)\in\MMMod{H}{}{}$ of
$\bbK$-linear maps from $V$ to $W$ is also left $H$-module by employing the adjoint action $\RA$
defined in (\ref{eqn:adjointaction}). In contrast to the endomorphisms we do not have an algebra structure
on homomorphisms. However, if $V,W\in\MMMod{H}{A}{A}$ are left $H$-module $A$-bimodules
with $A\in\AAAlg{H}{}{}$ then $\Hom_\bbK(V,W),\Hom_A(V,W)\in \MMMod{H}{A}{A}$ are left $H$-module
$A$-bimodules with $A$-actions given in (\ref{eqn:homAaction}).
For the left $A$-linear homomorphisms $_A\Hom(V,W)$ and their deformation we again refer to
\cite{ASCHIERISCHENKEL,SCHENKELDISS}.

Let now $H$ be a Hopf algebra with twist $\FF\in H\otimes H$. 
Similar to the endomorphisms we have two options to deform the homomorphisms:
Firstly, we can apply Theorem \ref{theo:algmoddef} and obtain the left 
$H^\FF$-module $A_\star$-bimodules $\Hom_\bbK(V,W)_\star,\Hom_A(V,W)_\star \in\MMMod{H^\FF}{A_\star}{A_\star}$.
Secondly, we have the homomorphisms between the deformed modules
$\Hom_\bbK(V_\star,W_\star),\Hom_{A_\star}(V_\star,W_\star)\in \MMMod{H^\FF}{A_\star}{A_\star}$.
Analogously to Theorem \ref{theo:quantendo} these modules are related by an isomorphism:
\begin{theo} \label{theo:quanthomo}
The map
\begin{flalign}
D_\FF : \Hom_\bbK(V,W)_\star \to \Hom_\bbK(V_\star,W_\star)~,\quad P\mapsto D_\FF(P) = (\bar f^\alpha\RA P)\circ\bar f_\alpha\ra
\end{flalign}
is a left $H^\FF$-module $A_\star$-bimodule isomorphism.
It restricts to a left $H^\FF$-module $A_\star$-bimodule isomorphism
\begin{flalign}
D_\FF : \Hom_A(V,W)_\star \to \Hom_{A_\star}(V_\star,W_\star)~.
\end{flalign}
\end{theo}
\begin{ex}
Let $V\in\MMMod{H}{A}{A}$. The dual module is given by 
$V^\prime := \Hom_A(V,A)\in\MMMod{H}{A}{A}$. The quantization map of Theorem \ref{theo:quanthomo}
leads to an isomorphism between the modules $(V^\prime)_\star= \Hom_A(V,A)_\star\in\MMMod{H^\FF}{A_\star}{A_\star}$
and $(V_\star)^\prime = \Hom_{A_\star}(V_\star,A_\star)\in\MMMod{H^\FF}{A_\star}{A_\star}$.
Thus, quantizing the dual module is (up to isomorphism) the same as dualizing the quantized module.
\end{ex}


\section{\label{sec:con}Quantization of connections}
We extend the results on the quantization isomorphism $D_\FF$ obtained in Section \ref{sec:hom} for endomorphisms and
homomorphisms to connections on modules. For this we require some notations from the theory of connections, see
e.g.~the lecture notes \cite{DuViLecture} for an introduction.


\subsection{Connections on modules}
\begin{defi}\label{defi:diffcalc}
Let $A$ be an associative and unital algebra. A {\bf differential calculus} over
$A$ (or a {\bf differential graded algebra}) is an $\bbN^0$-graded algebra 
$(\Omega^\bullet = \bigoplus_{n\geq 0} \Omega^n,\wedge)$,
 where $\Omega^0=A$, together with a $\bbK$-linear map $\dd:\Omega^\bullet\to\Omega^\bullet$
of degree one, satisfying $\dd\circ\dd=0$ and
\begin{flalign}
\dd(\omega\wedge\omega^\prime) = (\dd\omega)\wedge \omega^\prime + (-1)^{\deg(\omega)}\,\omega\wedge(\dd\omega^\prime)~,
\end{flalign}
for all $\omega,\omega^\prime\in\Omega^\bullet$ with $\omega$ of homogeneous degree.
\end{defi}
\sk
In the hypotheses above, the spaces $\Omega^n\in\MMMod{}{A}{A}$ are $A$-bimodules, for all $n\in \bbN^0$. 
We denote $\Omega^1$ simply by $\Omega$.
In analogy to classical
differential geometry, we call $\Omega^n$ the module of $n$-forms and
denote the product by a wedge $\wedge$. One has to be a bit careful with this notation, since in contrast to the
differential geometric wedge product, our product is {\it not} necessarily graded commutative.
\begin{ex}\label{ex:dga}
The exterior algebra of differential forms on a $d$-dimensional smooth manifold $M$ (together with the exterior differential)
is a differential calculus according to Definition \ref{defi:diffcalc}. There $A=C^\infty(M)$ and $\Omega^n$ 
is the module of smooth $n$-forms. We have $\Omega^n=0$ for $n>d$ and (as a special case) the 
wedge is graded commutative, 
for all $\omega,\omega^\prime\in\Omega^\bullet$ of homogeneous degree, $\omega\wedge\omega^\prime
= (-1)^{\deg(\omega)\,\deg(\omega^\prime)}\,\omega^\prime\wedge\omega$.
\end{ex}
\begin{defi}
Let $A$ be an associative and unital algebra and $(\Omega^\bullet,\wedge,\dd)$ be a differential calculus over $A$.
Let further $V\in\MMMod{}{}{A}$ be a right $A$-module.
A {\bf connection} on $V$ is a $\bbK$-linear map $\nab:V\to V\otimes_A\Omega$
satisfying the Leibniz rule, for all $v\in V$ and $a\in A$,
\begin{flalign}\label{eqn:leibnizcon}
\nab(v\cdot a) = \nab(v)\cdot a + v\otimes_A\dd a~.
\end{flalign}
We denote the set of all connections on $V$ by $\Con_A(V)$.
\end{defi}
\sk
Notice that $\Con_A(V)$ is an affine space over $\Hom_A(V,V\otimes_A\Omega)$: Let
$\nab\in\Con_A(V)$ and $P\in\Hom_A(V,V\otimes_A\Omega)$, then for all $v\in V$ and $a\in A$,
\begin{flalign}
(\nab+ P)(v\cdot a)  = \nab(v)\cdot a +v\otimes_A\dd a + P(v)\cdot a = (\nab+P)(v)\cdot a + v\otimes_A \dd a~. 
\end{flalign}
The action of $\Hom_A(V,V\otimes_A\Omega)$ on $\Con_A(V)$ is clearly free and transitive.

In case $V\in\MMMod{}{A}{A}$ is an $A$-bimodule a {\bf right connection} on $V$ is a $\bbK$-linear map 
$\nab:V\to V\otimes_A\Omega$ satisfying (\ref{eqn:leibnizcon}).
The space of right connections on $V$ is also denoted by $\Con_A(V)$. 
Notice that we do {\it not} demand any compatibility conditions for a right connection with the
left $A$-module structure on $V$. 
This is different to the concept of bimodule connections proposed in \cite{Mourad:1994xa,DuViMasson,DuViLecture}.
We will comment more on this later in Section \ref{sec:conlift} and argue that 
these conditions are too restrictive in our setting.

For connections on left $A$-modules or left connections on $A$-bimodules and their quantization we refer to
\cite{ASCHIERISCHENKEL,SCHENKELDISS}.
 
\subsection{Quantization isomorphism}
Let $H$ be a Hopf algebra and $A\in\AAAlg{H}{}{}$. 
Let further $(\Omega^\bullet,\wedge,\dd)$ be a {\bf left $H$-covariant differential calculus} over $A$, 
i.e.~$(\Omega^\bullet,\wedge)\in\AAAlg{H}{}{}$ is a left $H$-module algebra, the action $\ra$ is of degree zero
and the differential is equivariant, for all $\xi\in H$ and $\omega\in\Omega^\bullet$,
\begin{flalign}
\xi\ra(\dd \omega) = \dd(\xi\ra \omega)~,
\end{flalign}
or equivalently $\xi\RA\dd =\epsilon(\xi)\,\dd$. As a consequence, the modules of $n$-forms 
have the following structure $\Omega^n\in\MMMod{H}{A}{A}$, for all $n\in\bbN$.

Given a twist, a left $H$-covariant differential calculus is deformed into a left $H^\FF$-covariant differential calculus:
\begin{lem}\label{lem:dgadef}
Let $H$ be a Hopf algebra with twist $\FF\in H\otimes H$, $A\in\AAAlg{H}{}{}$ and $(\Omega^\bullet,\wedge,\dd)$
be a left $H$-covariant differential calculus over $A$. Then $(\Omega^\bullet,\wedge_\star,\dd)$ is a left $H^\FF$-covariant
differential calculus over $A_\star$, where $\wedge_\star$ is the deformed product obtained in Theorem \ref{theo:algmoddef}.
\end{lem}

Consider also a left $H$-module right $A$-module $V\in\MMMod{H}{}{A}$. If $V\in\MMMod{H}{A}{A}$ we simply
 forget about the left $A$-module structure.
Since $\Con_A(V)\subseteq \Hom_\bbK(V,V\otimes_A\Omega)$ we can apply the quantization
isomorphism $D_\FF$ of Theorem \ref{theo:quanthomo} to any connection $\nab\in\Con_A(V)$ in order to obtain a map
$D_\FF(\nab)\in\Hom_\bbK(V_\star,(V\otimes_A\Omega)_\star)$. This can not yet be a deformed connection, 
i.e.~a $\bbK$-linear map $V_\star\to V_\star\otimes_{A_\star}\Omega_\star$ satisfying the Leibniz rule (\ref{eqn:leibnizcon})
with respect to $\star$-multiplication, since the codomain of $D_\FF(\nab)$ is not $V_\star\otimes_{A_\star}\Omega_\star$
but $(V\otimes_A \Omega)_\star$. 
\begin{lem}\label{lem:varphi}
Let $H$ be a Hopf algebra with twist $\FF\in H\otimes H$ and let $A\in\AAAlg{H}{}{}$, $V\in\MMMod{H}{}{A}$
and $W\in\MMMod{H}{A}{A}$. Then the $\bbK$-linear map $\varphi:V_\star\otimes_{A_\star}W_\star\to(V\otimes_A W)_\star$
defined by, for all $v\in V_\star$ and $w\in W_\star$, $\varphi(v\otimes_{A_\star}w) := (\bar f^\alpha\ra v)\otimes_A(\bar f_\alpha\ra w)$
is a well-defined left $H^\FF$-module right $A_\star$-module isomorphism.
\end{lem}
 The strategy to obtain the quantization map $\widetilde{D}_\FF$ for connections is 
 to compose $D_\FF$ with $\varphi^{-1}$: 
 \begin{flalign}
 \xymatrix{
 V_\star \ar[rr]^-{\widetilde{D}_\FF(\nab)}\ar[rrd]_-{D_\FF(\nab)} & & V_\star\otimes_{A_\star} \Omega_\star\\
 && (V\otimes_A \Omega)_\star \ar[u]_-{\varphi^{-1}}
 }
 \end{flalign}
 To be precise, we define the left $H^\FF$-module isomorphism 
\begin{flalign}\label{eqn:Dtilde}
\widetilde{D}_\FF:\Hom_\bbK(V,V\otimes_A\Omega)_\star \to \Hom_\bbK(V_\star,V_\star\otimes_{A_\star}\Omega_\star)~,\quad P\mapsto \widetilde{D}_\FF(P) = \varphi^{-1}\circ D_\FF(P)~.
\end{flalign}
This map induces the quantization isomorphism for connections:
\begin{theo}\label{theo:conquant}
The left $H^\FF$-module isomorphism (\ref{eqn:Dtilde}) restricts to the left $H^\FF$-module isomorphism
\begin{flalign}
\widetilde{D}_\FF:\Hom_A(V,V\otimes_A\Omega)_\star \to \Hom_{A_\star}(V_\star,V_\star\otimes_{A_\star}\Omega_\star) 
\end{flalign}
and to the affine space isomorphism
\begin{flalign}
\widetilde{D}_\FF:\Con_A(V)\to \Con_{A_\star}(V_\star)~,
\end{flalign}
where $\Con_A(V)$ and $\Con_{A_\star}(V_\star)$ are respectively affine spaces over the isomorphic $H^\FF$-modules 
$\Hom_A(V,V\otimes_A\Omega)_\star$ and $\Hom_{A_\star}(V_\star,V_\star\otimes_{A_\star}\Omega_\star)$.
\end{theo}
The proof of this theorem can be found in \cite{ASCHIERISCHENKEL,SCHENKELDISS}.

\section{\label{sec:homlift}Tensor product of module homomorphisms}
Given two $\bbK$-modules $V$ and $W$ we can consider the tensor product $V\otimes W$ (over $\bbK$).
By construction, $V\otimes W$ is again a $\bbK$-module and we denote the
image of $(v,w)\in V\times W$ under the canonical $\bbK$-bilinear map 
$V\times W\to V\otimes W$ by $v\otimes w$.
Given two further $\bbK$-modules $\widetilde{V}$, $\widetilde{W}$ and also 
two $\bbK$-linear maps $P:V\to \widetilde{V}$ and $Q:W\to\widetilde{W}$ there is the tensor product map
$P\otimes Q:V\otimes W\to \widetilde{V}\otimes\widetilde{W}$ defined by, for all $v\in V$ and $w\in W$,
\begin{flalign}\label{eqn:naivetensor}
P\otimes Q(v\otimes w) :=P(v)\otimes Q(w)
\end{flalign}
 and extended to all $V\otimes W$ by $\bbK$-linearity.
This tensor product of $\bbK$-linear maps is associative, i.e.~for $Z$, $\widetilde{Z}$ being $\bbK$-modules and
$T:Z\to\widetilde{Z}$ being $\bbK$-linear, we have on $V\otimes W\otimes Z$
\begin{flalign}
P\otimes (Q\otimes T) =(P\otimes Q)\otimes T~.
\end{flalign}
It also satisfies the composition law, for $\widehat{V}$, $\widehat{W}$ being $\bbK$-modules 
and $\widetilde{P}:\widetilde{V}\to\widehat{V}$, $\widetilde{Q}:\widetilde{W}\to\widehat{W}$ being $\bbK$-linear,
\begin{flalign}
(\widetilde{P}\otimes\widetilde{Q})\circ (P\otimes Q) = (\widetilde{P}\circ P)\otimes (\widetilde{Q}\circ Q)~.
\end{flalign}

Let $H$ be a Hopf algebra and let us consider left $H$-modules $V,W\in\MMMod{H}{}{}$. Then we
 can define a left $H$-module structure on $V\otimes W$ by, for all $\xi\in H$, $v\in V$ and $w\in W$,
 \begin{flalign}\label{eqn:Hactiontensor}
 \xi\ra(v\otimes w) := (\xi_1\ra v)\otimes (\xi_2\ra w)~.
 \end{flalign}
The first question that now arises is the following: Is the tensor product $\otimes$ of $\bbK$-linear maps
compatible with this $H$-action? To answer this question we calculate the adjoint action of $\xi\in H$ on $P\otimes Q$ and
find
\begin{flalign}\label{eqn:Hactionnaivetensor}
\xi\RA(P\otimes Q) = (\xi_1\ra\circ P \circ S(\xi_3)\ra\, )\otimes (\xi_2\RA Q)~.
\end{flalign}
In case of a noncocommutative Hopf algebra, i.e.~$\xi_1\otimes \xi_2 \neq \xi_2\otimes \xi_1$,
and a nonequivariant map $Q$, i.e.~$\xi\RA Q \neq \epsilon(\xi)\,Q$, the right hand side of
(\ref{eqn:Hactionnaivetensor}) is not equal to $(\xi_1\RA P)\otimes (\xi_2\RA Q)$, meaning that 
the tensor product of $\bbK$-linear maps is not compatible with the $H$-action.
This can be understood as follows: The ordering on the left hand side of (\ref{eqn:naivetensor}) is $P,Q,v,w$, while
on the right hand side we have $P,v,Q,w$, thus $Q$ and $v$ do not appear properly ordered.
 This issue can be resolved if the Hopf algebra comes with a quasitriangular structure, which allows us to
 define braiding maps interchanging in a controlled way the ordering of $Q$ and $v$.

\subsection{Quasitriangular Hopf algebras}
\begin{defi}
A {\bf quasi-cocommutative Hopf algebra} $(H,\RR)$ is a Hopf algebra $H$ together with an invertible
element $\RR\in H\otimes H$ (called universal $R$-matrix) such that, for all $\xi\in H$,
\begin{flalign}\label{eqn:R1}
\Delta^\cop(\xi) = \RR\,\Delta(\xi)\,\RR^{-1}~,
\end{flalign} 
where $\Delta^\cop(\xi)=\xi_2\otimes\xi_1$ is the coopposite coproduct.

A quasi-cocommutative Hopf algebra is called {\bf quasitriangular} if
\begin{flalign}\label{eqn:R2}
(\Delta\otimes \id)\RR = \RR_{13}\,\RR_{23}~,\quad (\id\otimes\Delta)\RR=\RR_{13}\,\RR_{12}~,
\end{flalign}
and {\bf triangular} if additionally
$
\RR_{21} = \RR^{-1}
$, where $\RR_{21} = \sigma(\RR)$ with the transposition map $\sigma(\xi\otimes\zeta) = \zeta\otimes \xi$. 
\end{defi}
\sk
\begin{ex}
The Hopf algebra $U\Xi$ of diffeomorphisms on a smooth manifold $M$ is triangular with
trivial $R$-matrix $\RR=1\otimes 1$. The Hopf algebra of twist deformed diffeomorphisms $U\Xi^\FF$ 
is triangular with $R$-matrix $\RR^\FF = \FF_{21}\,\FF^{-1}$.

In general, the twist deformation of a (quasi)triangular Hopf algebra $(H,\RR)$ is (quasi)triangular
$(H^\FF,\RR^\FF:=\FF_{21}\,\RR\,\FF^{-1})$.
\end{ex}
\sk

Standard properties of quasitriangular $R$-matrices are (see e.g.~\cite{Majid:1996kd})
\begin{subequations}\label{eqn:Rprop}
\begin{flalign}
 \qquad\,\qquad\,\qquad\,\qquad&(\epsilon\otimes\id)\RR =1~,\quad& &(\id\otimes \epsilon)\RR=1~,&\qquad\,\qquad\\
 \qquad\,\qquad\,\qquad\,\qquad&(S\otimes\id)\RR =\RR^{-1}~,\quad& &(\id\otimes S)\RR^{-1}=\RR~,&\qquad\,\qquad
\end{flalign}
and the quantum Yang-Baxter equation
\begin{flalign}
\RR_{12}\,\RR_{13}\,\RR_{23} = \RR_{23}\,\RR_{13}\,\RR_{12}~.
\end{flalign}
\end{subequations}
It is convenient to introduce the notations $\RR = R^\alpha\otimes R_\alpha$ and 
$\RR^{-1} = \bar R^\alpha\otimes \bar R_\alpha$ (sum over $\alpha$ understood)
for the $R$-matrix and its inverse. 

Let $(H,\RR)$ be a quasitriangular Hopf algebra and $V,W\in\MMMod{H}{}{}$ be left $H$-modules.
Then we can define a left $H$-module isomorphism (called {\bf braiding}) between the left $H$-modules
$V\otimes W$ and $W\otimes V$. The braiding map $\tau_\RR: V\otimes W\to W\otimes V$
is defined by, for all $v\in V$ and $w\in W$,
\begin{flalign}\label{eqn:braiding}
\tau_\RR(v\otimes w) := (\bar R^\alpha\ra w)\otimes (\bar R_\alpha\ra v)~.
\end{flalign}
The property (\ref{eqn:R1}) ensures that $\tau_\RR$ is a left $H$-module homomorphism.
Furthermore, it is invertible via the map $\tau_\RR^{-1}$ given by 
$\tau^{-1}_\RR(w\otimes v) = (R_\alpha\ra v)\otimes (R^\alpha\ra w)$.
From the quasitriangular properties (\ref{eqn:R2}) we obtain the braid relations on the triple tensor product
$V\otimes W\otimes Z$ (with $Z\in\MMMod{H}{}{}$)
\begin{flalign}
\tau_{\RR\,(12)3} = \tau_{\RR\, 12}\circ \tau_{\RR \, 23}~,\quad \tau_{\RR\,1(23)} = \tau_{\RR\,23}\circ \tau_{\RR\,12}~,
\end{flalign}
where the indices on $\tau_\RR$ label the legs of the tensor product on which $\tau_\RR$ is acting on, 
e.g.~$\tau_{\RR\,23}(v\otimes w\otimes z) = v\otimes (\bar R^\alpha\ra z)\otimes (\bar R_\alpha\ra w)$
and $\tau_{\RR\,(12)3}(v\otimes w\otimes z) = (\bar R^\alpha\ra z)\otimes \bar R_\alpha\ra(v\otimes w)$.

\subsection{Tensor product over $\bbK$}
In order to resolve the $H$-noncovariance of the usual tensor product
of $\bbK$-linear maps (\ref{eqn:naivetensor}) we make the following
\begin{defi}
Let $(H,\RR)$ be a quasitriangular Hopf algebra and $V,W,\widetilde{V},\widetilde{W}\in\MMMod{H}{}{}$.
The {\bf $R$-tensor product} of $\bbK$-linear maps is defined by, for all
$P\in\Hom_\bbK(V,\widetilde{V})$ and $Q\in\Hom_\bbK(W,\widetilde{W})$,
\begin{flalign}\label{eqn:Rtensor}
P\otimes_\RR Q := (P\circ \bar R^\alpha\ra\,)\otimes (\bar R_\alpha\RA Q)~,
\end{flalign}
where $\otimes$ is defined in (\ref{eqn:naivetensor}).
\end{defi}
\sk
This product already appeared in \cite{Majid:1996kd}, Chapter 9.3,  where however only some of its properties
have been studied. In particular, the twist deformation to be studied later in this subsection and the reduction 
to tensor products over $A$ in the next subsection is to our best knowledge new material, which is
 for the first time presented in \cite{ASCHIERISCHENKEL,SCHENKELDISS}.

We can rewrite (\ref{eqn:Rtensor}) in terms of the braiding map (\ref{eqn:braiding})
and its inverse as follows
\begin{flalign}
P\otimes_\RR Q = (P\otimes_\RR\id)\circ (\id\otimes_\RR Q)=(P\otimes\id)\circ \tau_\RR\circ (Q\otimes\id)\circ \tau_\RR^{-1}~.
\end{flalign}
From this expression we obtain that the lift of $P$ to the tensor product module
$V\otimes \widetilde{W}$ is simply $P\mapsto P\otimes_\RR\id = P\otimes\id$, while the lift 
of $Q$, given by $Q\mapsto \id\otimes_\RR Q = \tau_\RR\circ (Q\otimes \id)\circ\tau_\RR^{-1}$, requires
a braiding of the module $W$ to the very left before acting with $Q$.

The ordering problem (\ref{eqn:Hactionnaivetensor}) is resolved for the $R$-tensor product
and furthermore $\otimes_\RR$ is associative and satisfies a braided composition law.
\begin{propo}\label{propo:Rtensorproperties}
Let $(H,\RR)$ be a quasitriangular Hopf algebra and $V,W,Z,\widetilde{V},\widetilde{W},\widetilde{Z},\widehat{V}
,\widehat{W}\in\MMMod{H}{}{}$ be left $H$-modules. 
The $R$-tensor product is compatible with the left $H$-module structure, i.e., for all
$\xi\in H$, $P\in\Hom_\bbK(V,\widetilde{V})$ and $Q\in\Hom_\bbK(W,\widetilde{W})$,
\begin{subequations}
\begin{flalign}
\xi\RA(P\otimes_\RR Q) =( \xi_1\RA P)\otimes_\RR (\xi_2\RA Q)~.
\end{flalign}
Furthermore, $\otimes_\RR$ is associative, for all $P\in\Hom_\bbK(V,\widetilde{V})$,
$Q\in\Hom_\bbK(W,\widetilde{W})$ and $T\in\Hom_\bbK(Z,\widetilde{Z})$,
\begin{flalign}
(P\otimes_\RR Q)\otimes_\RR T = P\otimes_\RR(Q\otimes_\RR T)~,
\end{flalign}
and satisfies the composition law, for all $P\in\Hom_\bbK(V,\widetilde{V})$,
$Q\in\Hom_\bbK(W,\widetilde{W})$, $\widetilde{P}\in\Hom_\bbK(\widetilde{V},\widehat{V})$
and $\widetilde{Q}\in\Hom_\bbK(\widetilde{W},\widehat{W})$,
\begin{flalign}\label{eqn:compo}
(\widetilde{P}\otimes_\RR\widetilde{Q})\circ (P\otimes_\RR Q) = \big(\widetilde{P}\circ (\bar R^\alpha\RA P)\big)
\otimes_\RR \big( (\bar R_\alpha\RA \widetilde{Q})\circ Q\big)~.
\end{flalign}
\end{subequations}
\end{propo}
The tensor product $\otimes_\RR$ is also compatible with right $A$-linearity:
\begin{propo}
In the hypotheses of Proposition \ref{propo:Rtensorproperties} let $W,\widetilde{W}\in\MMMod{H}{}{A}$
be left $H$-module right $A$-modules with $A\in\AAAlg{H}{}{}$. Then for all
$P\in\Hom_\bbK(V,\widetilde{V})$ and all right $A$-linear $Q\in\Hom_A(W,\widetilde{W})$ we have
\begin{flalign}
P\otimes_\RR Q\in\Hom_A(V\otimes W,\widetilde{V}\otimes \widetilde{W})~.
\end{flalign}
\end{propo}
A proof of these propositions is given in \cite{ASCHIERISCHENKEL,SCHENKELDISS}.

It remains to study the behavior of $\otimes_\RR$ under twist quantization.
Let $\FF\in H\otimes H$ be a twist of the quasitriangular Hopf algebra $(H,\RR)$.
Then the deformed Hopf algebra $H^\FF$ is again quasitriangular with $R$-matrix
$\RR^\FF = \FF_{21}\,\RR\,\FF^{-1}$.
Given two left $H$-modules $V,W\in\MMMod{H}{}{}$, then $V$ and $W$ are automatically
left $H^\FF$-modules, since $H$ and $H^\FF$ are equal as algebras, $V,W\in\MMMod{H^\FF}{}{}$.
Analogously, $V\otimes W\in\MMMod{H}{}{}$ is also a left $H^\FF$-module.
Notice however that the definition of the left $H$-action (and thus also the left $H^\FF$-action)
on $V\otimes W$ uses the coproduct in $H$, see (\ref{eqn:Hactiontensor}).
We can induce another left $H^\FF$-action (and thus also a left $H$-action) on the
$\bbK$-module $V\otimes W$  by using the coproduct in $H^\FF$. To distinguish
the resulting left $H^\FF$-module from $V\otimes W\in\MMMod{H^\FF}{}{}$
we denote it by $V\otimes_\star W\in\MMMod{H^\FF}{}{}$. 
The image of $(v,w)$ under the canonical map $V\times W\to V\otimes_\star W$
is denoted by $v\otimes_\star w$. By definition, the left $H^\FF$-action reads, 
for all $\xi\in H^\FF$, $v\in V$ and $w\in W$,
\begin{flalign}
\xi\ra_\FF(v\otimes_\star w) = (\xi_{1_\FF}\ra v )\otimes_\star (\xi_{2_\FF}\ra w)~,
\end{flalign}
where we have used the notation $\Delta^\FF(\xi)=\xi_{1_\FF}\otimes\xi_{2_\FF}$.

The left $H^\FF$-modules $V\otimes W\in\MMMod{H^\FF}{}{}$ and $V\otimes_\star W\in\MMMod{H^\FF}{}{}$
are isomorphic via the map $\varphi: V\otimes_\star W \to V\otimes W$ defined by, for all $v\in V$ and $w\in W$,
\begin{flalign}
\varphi(v\otimes_\star w) := (\bar f^\alpha\ra v) \otimes (\bar f_\alpha\ra w)~.
\end{flalign}
Compare this also with Lemma \ref{lem:varphi} which discusses the corresponding induced
 map on the quotient modules (i.e.~tensor product over $A$, respectively $A_\star$).
 We obtain the following deformation behavior of the $R$-tensor product:
 \begin{theo}\label{theo:deformationproducthom}
 Let $(H,\RR)$ be a quasitriangular Hopf algebra with twist $\FF\in H\otimes H$ 
 and $V,W,\widetilde{V},\widetilde{W}\in\MMMod{H}{}{}$.
 Then for all $P\in\Hom_\bbK(V,\widetilde{V})$ and $Q\in\Hom_\bbK(W,\widetilde{W})$ the following
 diagram commutes
 \begin{flalign}
 \xymatrix{
V\otimes_\star W \ar[d]_-{\varphi}\ar[rrrr]^-{D_\FF(P)\,\otimes_{\RR^\FF} \,D_\FF(Q)}&&&& \widetilde{V}\otimes_\star\widetilde{W}\ar[d]^-{\varphi}\\
V\otimes W \ar[rrrr]_-{D_\FF\big( (\bar f^\alpha\RA P)\,\otimes_\RR\, (\bar f_\alpha\RA Q) \big)} &&&& \widetilde{V}\otimes \widetilde{W}
 }
 \end{flalign}
 where $\otimes_{\RR^\FF}$ is the $R$-tensor product with respect to the Hopf algebra $(H^\FF,\RR^\FF)$ 
 and $D_\FF$ is the quantization isomorphism of Theorem \ref{theo:quanthomo}.
 \end{theo}
 The proof can be found in \cite{ASCHIERISCHENKEL,SCHENKELDISS}.

\subsection{Tensor product over $A$}
Let $(H,\RR)$ be a quasitriangular Hopf algebra, $A\in\AAAlg{H}{}{}$, $V,\widetilde{V}\in\MMMod{H}{}{A}$
and $W,\widetilde{W}\in\MMMod{H}{A}{A}$.
We now investigate under  which conditions the $R$-tensor product of $\bbK$-linear maps
(\ref{eqn:Rtensor}) gives rise to a tensor product of right $A$-linear maps
$P\in\Hom_A(V,\widetilde{V})$ and $Q\in\Hom_A(W,\widetilde{W})$ on $V\otimes_A W$.
Remember that $V\otimes_A W$ can be defined as the quotient of $V\otimes W$ by the submodule
generated by the elements $v\cdot a\otimes w - v\otimes a\cdot w $, for all $v\in V$, $w\in W$ and $a\in A$.

It turns out that we require in addition to the quasitriangular structure on $H$ a quasi-commutative structure
on $A$, $W$ and $\widetilde{W}$.
\begin{defi}
Let $(H,\RR)$ be a quasitriangular Hopf algebra.
We say that a left $H$-module algebra $A\in\AAAlg{H}{}{}$ is {\bf quasi-commutative} if, for all $a,b\in A$,
\begin{flalign}
a\,b = (\bar R^\alpha\ra b)\,(\bar R_\alpha\ra a)~.
\end{flalign}
A left $H$-module $A$-bimodule $W\in\MMMod{H}{A}{A}$ is {\bf quasi-commutative} if, for all $w\in W$ and $a\in A$,
\begin{flalign}\label{eqn:modqc}
v\cdot a = (\bar R^\alpha\ra a)\cdot(\bar R_\alpha\ra v)~.
\end{flalign}
\end{defi}
\sk
It can be shown that tensor products $W\otimes_A \widetilde{W}$ of, and homomorphisms
$\Hom_A(W,\widetilde{W})$ between quasi-commutative left $H$-module $A$-bimodules $W,\widetilde{W}$
over quasi-commutative  left $H$-module algebras $A$ are quasi-commutative left $H$-module $A$-bimodules 
\cite{ASCHIERISCHENKEL,SCHENKELDISS}.
\begin{ex}
Let $U\Xi$ be the Hopf algebra of diffeomorphisms on a manifold $M$. The $R$-matrix in this case
is $\RR=1\otimes 1$, i.e.~$U\Xi$ is cocommutative. The algebra of functions $C^\infty(M)$
is a quasi-commutative left $U\Xi$-module algebra and the left $U\Xi$-module $C^\infty(M)$-bimodules
of one-forms $\Omega$ and vector fields $\Xi$ are quasi-commutative as well.

The twist deformed algebra of functions $C^\infty(M)_\star$ and the twist quantized modules
of one-forms $\Omega_\star$ and vector fields $\Xi_\star$ are quasi-commutative with
 respect to the deformed Hopf algebra of diffeomorphisms
$U\Xi^\FF$ and the $R$-matrix $\RR^\FF = \FF_{21}\,\FF^{-1}$.

In general, twist deformations of quasi-commutative algebras and bimodules are again quasi-commutative.
\end{ex}
\sk
Using quasi-commutativity we can restrict the $R$-tensor product to tensor product modules over $A$.
\begin{theo}
Let $(H,\RR)$ be a quasitriangular Hopf algebra, $A\in\AAAlg{H}{}{}$, $V,\widetilde{V}\in\MMMod{H}{}{A}$
and $W,\widetilde{W}\in\MMMod{H}{A}{A}$. Let also $A$, $W$ and $\widetilde{W}$ be quasi-commutative.
Then for all $P\in\Hom_A(V,\widetilde{V})$ and $Q\in\Hom_A(W,\widetilde{W})$ the 
map $P\otimes_\RR Q \in\Hom_A(V\otimes W,\widetilde{V}\otimes \widetilde{W})$ defined in (\ref{eqn:Rtensor})
induces a well-defined map between the quotient modules (which we denote with a slight abuse of notation by the same symbol)
\begin{flalign}
P\otimes_\RR Q \in\Hom_A(V\otimes_A W,\widetilde{V}\otimes_A\widetilde{W})~.
\end{flalign}
Explicitly, we have for all $v\in V$ and $w\in W$,
\begin{flalign}
P\otimes_\RR Q (v\otimes_A w) = P(\bar R^\alpha\ra v)\otimes_A (\bar R_\alpha\RA Q)(w)~.
\end{flalign}
This tensor product of right $A$-linear maps is compatible with the left $H$-action, it is associative and it
satisfies the composition property (\ref{eqn:compo}).
\end{theo}
Again, the proof is found in \cite{ASCHIERISCHENKEL} and for the special case of triangular Hopf algebras in \cite{SCHENKELDISS}.
There we also studied the twist quantization and an analogous commuting diagram as in Theorem
\ref{theo:deformationproducthom} is found for the $R$-tensor product of right $A$-linear maps:
\begin{flalign}
 \xymatrix{
V_\star\otimes_{A_\star} W_\star \ar[d]_-{\varphi}\ar[rrrr]^-{D_\FF(P)\,\otimes_{\RR^\FF} \,D_\FF(Q)}&&&& \widetilde{V}_\star \otimes_{A_\star}\widetilde{W}_\star\ar[d]^-{\varphi}\\
(V\otimes_A W)_\star \ar[rrrr]_-{D_\FF\big( (\bar f^\alpha\RA P)\,\otimes_\RR\, (\bar f_\alpha\RA Q) \big)} &&&& (\widetilde{V}\otimes_A \widetilde{W})_\star
 }
 \end{flalign}


\section{\label{sec:conlift}Connections on tensor product modules}
Motivated by the investigations in the last section we now study the extension of
connections to tensor product modules. This is of particular importance in noncommutative gravity, since
it allows us to define connections on noncommutative tensor fields in terms of a fundamental connection 
on vector fields or one-forms.

We again require a quasitriangular Hopf algebra $(H,\RR)$ and a quasi-commutative 
left $H$-module algebra $A\in\AAAlg{H}{}{}$. We further assume a {\bf graded quasi-commutative}
left $H$-covariant differential calculus $(\Omega^\bullet,\wedge,\dd)$ over $A$, 
i.e.~for all $\omega,\omega^\prime\in\Omega^\bullet$ of homogeneous degree we demand
\begin{flalign}
\omega\wedge \omega^\prime = (-1)^{\deg(\omega)\,\deg(\omega^\prime)}\,(\bar R^\alpha\ra \omega^\prime)\wedge (\bar R_\alpha\ra \omega)~.
\end{flalign}
An example for a graded quasi-commutative  left $H$-covariant differential calculus
can be obtained by deforming the classical differential calculus on a manifold (see Example
\ref{ex:dga}) by using Lemma \ref{lem:dgadef}. This is the differential calculus of noncommutative
gravity.

Before studying the extension of connections to tensor product modules it is instructive 
to present and discuss the following
\begin{lem}\label{lem:braidedleftleibniz}
Let $(H,\RR)$ be a quasitriangular Hopf algebra, $A\in\AAAlg{H}{}{}$ be quasi-commutative
 and let $(\Omega^\bullet,\wedge,\dd)$ be a graded quasi-commutative left 
 $H$-covariant differential calculus over $A$. Then any right connection 
 $\nab\in\Con_A(W)$ on a quasi-commutative left $H$-module $A$-bimodule
 $W\in\MMMod{H}{A}{A}$ satisfies the braided  left Leibniz rule, for all $w\in W$ and
 $a\in A$,
 \begin{flalign}
 \nn \nab(a\cdot w) &= (\bar R^\alpha\ra a)\cdot (\bar R_\alpha\RA\nab)(w) + (R_\alpha\ra w)\otimes_A (R^\alpha\ra \dd a)\\
 \label{eqn:braidleftleib}&=(\bar R^\alpha\ra a)\cdot (\bar R_\alpha\RA\nab)(w) +\tau_\RR^{-1}(\dd a\otimes_A w)~.
 \end{flalign}
\end{lem}
\begin{proof}
The proof follows from quasi-commutativity (\ref{eqn:modqc}), the right Leibniz rule of the connections (\ref{eqn:leibnizcon}),
$H$-equivariance of the differential and the $R$-matrix properties (\ref{eqn:R2}) and (\ref{eqn:Rprop}),
\begin{flalign}
\nn\nab(a\cdot w) &= \nab\big((R_\alpha\ra w )\cdot (R^\alpha\ra a)\big) = \nab(R_\alpha\ra w)\cdot (R^\alpha\ra a) + (R_\alpha\ra w)\otimes_A \dd(R^\alpha\ra a)\\
\nn &=(\bar R^\beta R^\alpha\ra a)\cdot \bar R_\beta\ra \nab(R_\alpha\ra w) + (R_\alpha\ra w)\otimes_A (R^\alpha\ra \dd a)\\
\nn &=(\bar R^\beta \bar R^\alpha\ra a)\cdot \bar R_\beta\ra \nab(S(\bar R_\alpha)\ra w) + (R_\alpha\ra w)\otimes_A (R^\alpha\ra \dd a)\\
\nn&= ( \bar R^\alpha\ra a)\cdot \bar R_{\alpha_1}\ra \nab(S(\bar R_{\alpha_2})\ra w) + (R_\alpha\ra w)\otimes_A (R^\alpha\ra \dd a)\\
&= ( \bar R^\alpha\ra a)\cdot (\bar R_{\alpha}\RA \nab)( w) + (R_\alpha\ra w)\otimes_A (R^\alpha\ra \dd a)~.
\end{flalign}
\end{proof}
\begin{rem}
Lemma \ref{lem:braidedleftleibniz} shows that a generic right connection on $W$
is in general not a bimodule connection in the sense of  \cite{Mourad:1994xa,DuViMasson,DuViLecture}. Remember that 
a right bimodule connection is defined to be a right connection
on an  $A$-bimodule $W$, such that {\it additionally}, for all $a\in A$ and $w\in W$,
\begin{flalign}\label{eqn:duvicon}
\nab(a\cdot w) = a\cdot \nab(w) + \sigma(\dd a\otimes_A w)~,
\end{flalign}
with some generalized braiding $\sigma: \Omega\otimes_A W \to W\otimes_A \Omega$.
Comparing our expression (\ref{eqn:braidleftleib}) with this equation we identify $\sigma = \tau_\RR^{-1}$, but the first
term does not match. Notice that in case the $R$-matrix acts trivially on $\nab$, such that
$\bar R^\alpha \otimes \bar R_\alpha\RA\nab = 1\otimes \nab$, the two formulae coincide. This shows that
equivariant connections are automatically bimodule connections.
For nonequivariant connections the $R$-matrix in the first term of (\ref{eqn:braidleftleib}) 
is naturally interpreted as a braiding between $\nab$ and $a$, taking care of the reordering.
Notice that we {\it did not} demand (\ref{eqn:braidleftleib}), but it
is a consequence of quasi-commutativity and holds {\it for all} connections $\nab\in\Con_A(W)$!
\end{rem}
\sk
The main reason why in \cite{Mourad:1994xa,DuViMasson,DuViLecture} the connections are assumed to
be bimodule connections is that (\ref{eqn:duvicon}) and the braiding property of $\sigma$ allows
one to define connections on tensor product modules $V\otimes_A W$ in terms of connections
on $V$ and $W$. For equivariant connections we could use the same extension formula as in
 \cite{Mourad:1994xa,DuViMasson,DuViLecture}.
 
 In the following we show that in our setting there is a generalization of the tensor product connection 
 of \cite{Mourad:1994xa,DuViMasson,DuViLecture}, which is based on (\ref{eqn:braidleftleib})  and
  applies to {\it all} connections (in particular also nonequivariant ones).
\begin{theo}\label{theo:conext}
Let $(H,\RR)$ be a quasitriangular Hopf algebra and $A\in\AAAlg{H}{}{}$, $W\in\MMMod{H}{A}{A}$ be quasi-commutative.
Let further $(\Omega^\bullet,\wedge,\dd)$ be a graded quasi-commutative left $H$-covariant differential calculus over $A$
and $V\in\MMMod{H}{}{A}$, $\nab_V\in\Con_A(V)$ and $\nab_W\in\Con_A(W)$.
Then the $\bbK$-linear map $\nab_V\oplus_\RR\nab_W : V\otimes_A W\to V\otimes_AW\otimes_A\Omega$
defined by, for all $v\in V$ and $w\in W$,
\begin{flalign}\label{eqn:oplus}
(\nab_V\oplus_\RR\nab_W)(v\otimes_A w) := \tau^{-1}_{\RR\,23}\big(\nab_V(v)\otimes_A w\big) 
+(\bar R^\alpha\ra v) \otimes_A (\bar R_\alpha\RA \nab_W)(w)
\end{flalign}
is well-defined and a connection on $V\otimes_A W$.
\end{theo}
 Again, if we assume as a very special case that  $\nab_W$ is equivariant, then our formula (\ref{eqn:oplus}) reduces to the one in 
 \cite{Mourad:1994xa,DuViMasson,DuViLecture}. 
 The $R$-sum of connections (\ref{eqn:oplus}) extends canonically to higher tensor products of modules.
 \begin{theo}
 In the hypotheses of Theorem \ref{theo:conext} let also $Z\in\MMMod{H}{A}{A}$ be quasi-commutative and
 $\nab_Z\in\Con_A(Z)$. Then we have on $V\otimes_AW\otimes_A Z$
 \begin{flalign}
 (\nab_V\oplus_\RR \nab_W )\oplus_\RR\nab_Z = \nab_V\oplus_\RR(\nab_W\oplus_\RR\nab_Z)~.
 \end{flalign}
 \end{theo}
 The proof of the two theorems can be found in \cite{ASCHIERISCHENKEL} and
  for the special case of triangular Hopf algebras in \cite{SCHENKELDISS}.
 
To conclude this section we investigate the twist deformation of the $R$-sum of connections.
\begin{theo}
In the hypotheses of Theorem \ref{theo:conext} let $\FF\in H\otimes H$ be a twist. Then the following diagram
commutes
\begin{flalign}
\xymatrix{
V_\star\otimes_{A_\star} W_\star\ar[d]_-{\varphi} \ar[rrrr]^-{\widetilde{D}_\FF(\nab_V) \oplus_{\RR^\FF}\widetilde{D}_\FF(\nab_W)} &&&& V_\star\otimes_{A_\star} W_\star\otimes_{A_\star} \Omega_\star\ar[d]^-{\varphi}\\
(V\otimes_A W)_\star \ar[rrrr]_-{D_\FF(\nab_V\oplus_\RR\nab_W)} &&&& (V\otimes_AW\otimes_A \Omega)_\star
}
\end{flalign}
where $\varphi$ is the isomorphism of Lemma \ref{lem:varphi}, 
$\widetilde{D}_\FF$ is the quantization isomorphism for connections of Theorem \ref{theo:conquant}
and $\oplus_{\RR^\FF}$ is the $R$-sum of connections with respect to the quasitriangular Hopf algebra
$(H^\FF,\RR^\FF)$.
\end{theo}
Also this theorem is proven in \cite{ASCHIERISCHENKEL,SCHENKELDISS}.


\appendix
\section*{Acknowledgements}
I would like to thank the organizers and participants
of the  \textit{Corfu Summer Institute 2011} for 
this very interesting conference. In particular, I thank  Tomasz Brzezi{\'n}ski and
Walter van Suijlekom for helpful discussions.
Additionally,  I am grateful to Thorsten Ohl and Christoph F.~Uhlemann for
useful discussions and comments on this work. This project was supported 
by  Deutsche Forschungsgemeinschaft through the Research 
Training Group GRK\,1147 \textit{Theoretical Astrophysics and Particle Physics},
the ESF activity ``Quantum Geometry and Quantum Gravity'' via the short visit grant 3267
and  INFN Torino gruppo collegato di Alessandria.


\end{document}